\documentclass[a4paper,12pt,oneside,headsepline]{scrartcl}

\usepackage{etoolbox}
\usepackage{ifdraft}
\usepackage{iftex}

\ifluatex
\usepackage[utf8]{luainputenc}
\else
\usepackage[utf8]{inputenc}
\fi
\usepackage[T1]{fontenc}

\usepackage[english]{babel}

\usepackage{lmodern}
\usepackage{fourier}

\usepackage{amsfonts}
\usepackage{mathrsfs} %
\usepackage{dsfont}

\usepackage{amssymb}

\usepackage{amsmath}
\usepackage{mathtools}
\usepackage[thmmarks, amsmath, amsthm]{ntheorem}
\usepackage{nccmath}    %
\usepackage{tensor}

\usepackage[draft=false]{scrlayer-scrpage}
\usepackage{needspace}

\providebool{nobiblatex}
\boolfalse{nobiblatex}
\ifbool{nobiblatex}{%
}{%
  \usepackage[backend=biber,style=numeric-comp,giveninits,url=false,sortcites=true,maxbibnames=99]{biblatex}
  \addbibresource{bibliography.bib}
}

\usepackage[final,
            pdfborder={0 0 0}, colorlinks=true,
            linkcolor=BrickRed, citecolor=ForestGreen, urlcolor=RoyalBlue]{hyperref}

\usepackage[cmyk,dvipsnames]{xcolor}

\ifdraft{%
\usepackage[textsize=scriptsize,colorinlistoftodos]{todonotes}
\usepackage{lineno}
\usepackage{scrtime}
}{}

\usepackage{booktabs}
\usepackage[inline]{enumitem}
\usepackage{lettrine}
\usepackage{url}

\KOMAoptions{DIV=12}

\ifdraft{%
\linenumbers
\newcommand*\patchAmsMathEnvironmentForLineno[1]{%
  \expandafter\let\csname old#1\expandafter\endcsname\csname #1\endcsname
  \expandafter\let\csname oldend#1\expandafter\endcsname\csname end#1\endcsname
  \renewenvironment{#1}%
     {\linenomath\csname old#1\endcsname}%
     {\csname oldend#1\endcsname\endlinenomath}}%
\newcommand*\patchBothAmsMathEnvironmentsForLineno[1]{%
  \patchAmsMathEnvironmentForLineno{#1}%
  \patchAmsMathEnvironmentForLineno{#1*}}%
\AtBeginDocument{%
\patchBothAmsMathEnvironmentsForLineno{equation}%
\patchBothAmsMathEnvironmentsForLineno{align}%
\patchBothAmsMathEnvironmentsForLineno{flalign}%
\patchBothAmsMathEnvironmentsForLineno{alignat}%
\patchBothAmsMathEnvironmentsForLineno{gather}%
\patchBothAmsMathEnvironmentsForLineno{multline}%
}}

\clearmainofpairofpagestyles
\automark[section]{subsection}

\renewcommand{\subsectionmark}[1]{}

\lohead{{\small\normalfont \headertitle}}
\rohead{{\small \headerauthors}}

\RedeclareSectionCommand[%
  font=\Large\sffamily\bfseries,%
  beforeskip=1\baselineskip,%
  afterskip=0.5\baselineskip,%
  indent=0em,%
  afterindent=false,%
  tocbeforeskip=0.3\baselineskip plus 1pt minus 1pt
  ]{section}

\RedeclareSectionCommands[%
  font=\normalfont\bfseries,%
  beforeskip=3pt,%
  afterskip=-1em%
  ]{subsection,subsubsection}

\RedeclareSectionCommands[%
  font=\normalfont\itshape,%
  beforeskip=1pt,%
  afterskip=-1em,%
  indent=0pt,%
  ]{paragraph}

\ifbool{nobiblatex}{}{%
  \AtEveryBibitem{\clearfield{doi}}
  \AtEveryBibitem{\clearfield{isbn}}
  \AtEveryBibitem{\clearfield{issn}}
  \AtEveryBibitem{\clearfield{pages}}
  \AtEveryBibitem{\clearlist{language}}

  \setlength\bibitemsep{3pt}
  \renewcommand*{\bibfont}{\footnotesize}
  \renewbibmacro*{in:}{}
  \DeclareFieldFormat
    [article,inbook,incollection,inproceedings,patent,thesis,unpublished,misc]
    {title}{#1}
}

\newenvironment{enumerateroman}{
\begin{enumerate}[label=(\roman*),%
  leftmargin=2.5em,itemindent=0pt,%
  labelindent=.5em,labelwidth=1.5em,labelsep=!,%
  nosep]
}{
\end{enumerate}
}

\newenvironment{enumeratearabic*}{
\begin{enumerate*}[label=(\arabic*)] %
}{
\end{enumerate*}
}

\newenvironment{enumerateroman*}{
\begin{enumerate*}[label=(\roman*)] %
}{
\end{enumerate*}
}

\numberwithin{equation}{section}

\theoremnumbering{arabic}
\newtheorem{theoremcounter}{theoremcounter}[section]
\theoremnumbering{Alph}
\newtheorem{maintheoremcounter}{maintheoremcounter}

\theoremstyle{plain}

\newtheorem{corollary}[theoremcounter]{Corollary}
\newtheorem{lemma}[theoremcounter]{Lemma}

\newtheorem{proposition}[theoremcounter]{Proposition}
\newtheorem{theorem}[theoremcounter]{Theorem}

\theoremstyle{plain}

\newtheorem{maincorollary}[maintheoremcounter]{Corollary}
\newtheorem{maintheorem}[maintheoremcounter]{Theorem}

\theoremstyle{definition}

\newtheorem{definition}[theoremcounter]{Definition}
\newtheorem{example}[theoremcounter]{Example}

\theoremstyle{remark}

\newtheorem{remark}[theoremcounter]{Remark}

\theoremstyle{nonumberremark}

\newtheorem{mainremark}{Remark}

\newcommand{\tx}{\text}
\newcommand{\thdash}{\nbd th}
\newcommand{\nbd}{\nobreakdash-\hspace{0pt}}

\makeatletter
\newcommand{\writelabel}[1]{#1\def\@currentlabel{#1}}
\makeatother

\newcommand{\minwidthmathbox}[2]{%
  \mathmakebox[{\ifdim#1<\width\width\else#1\fi}]{#2}%
}

\newcommand{\tbf}{\bfseries}

\newcommand{\bbV}{\ensuremath{\mathbb{V}}}

\newcommand{\bbone}{\ensuremath{\mathds{1}}}

\newcommand{\cO}{\ensuremath{\mathcal{O}}}

\newcommand{\rmH}{\ensuremath{\mathrm{H}}}

\newcommand{\rmM}{\ensuremath{\mathrm{M}}}

\newcommand{\td}{\tilde}
\newcommand{\wtd}{\widetilde}
\newcommand{\ov}{\overline}

\newcommand{\defcol}{\mathrel{:}}
\newcommand{\defeq}{\mathrel{:=}}

\newcommand{\condsep}{\mathrel{\;:\;}}
\newcommand{\condcol}{\condsep}
\newcommand{\quantsep}{\mathrel{\;.\;}}

\newcommand{\mrelspace}[1]{\mathrel{\mspace{#1}}}

\let\rightarroworig\rightarrow
\renewcommand{\rightarrow}
  {\protect\relbar\mrelspace{-9.7mu}\rightarroworig}

\renewcommand{\twoheadrightarrow}
  {\protect\rightarroworig\mrelspace{-15mu}\rightarroworig}

\let\leftarroworig\leftarrow
\renewcommand{\leftarrow}
  {\protect\leftarroworig\mrelspace{-9.7mu}\relbar}

\renewcommand{\longrightarrow}
  {\protect\relbar\mrelspace{-3.2mu}\relbar\mrelspace{-9.5mu}\rightarroworig}
\newcommand{\longhookrightarrow}
  {\protect\lhook\mrelspace{-3.1mu}\relbar\mrelspace{-3.2mu}\relbar\mrelspace{-11.7mu}\rightarroworig}

\makeatletter
\@ifundefined{xlongrightarrow}{%

}{%

}
\makeatother

\makeatletter
\@ifundefined{xlongleftarrow}{%

}{%

}
\makeatother

\newcommand{\ra}{\rightarrow}

\newcommand{\thra}{\twoheadrightarrow}

\newcommand{\lra}{\longrightarrow}
\newcommand{\lhra}{\longhookrightarrow}

\newcommand{\mto}{\mapsto}
\newcommand{\lmto}{\longmapsto}

\renewcommand{\Im}{\mathrm{Im}}

\newcommand{\isdiv}{\mathrel{\mid}}

\renewcommand{\pmod}[1]{\;(\mathrm{mod}\, #1)}

\renewcommand{\ker}{\operatorname{ker}}

\newenvironment{psmatrix}{\left(\begin{smallmatrix}}{\end{smallmatrix}\right)}
\newcommand{\Mat}[1]{\operatorname{Mat}_{#1}}

\newcommand{\linspan}{\operatorname{span}}

\newcommand{\rk}{\operatorname{rk}}

\newcommand{\ZZ}{\ensuremath{\mathbb{Z}}}
\newcommand{\QQ}{\ensuremath{\mathbb{Q}}}
\newcommand{\RR}{\ensuremath{\mathbb{R}}}
\newcommand{\CC}{\ensuremath{\mathbb{C}}}

\newcommand{\GL}[1]{\ensuremath{\mathrm{GL}_{#1}}}
\newcommand{\Mp}[1]{\ensuremath{\mathrm{Mp}_{#1}}}

\newcommand{\SL}[1]{\ensuremath{\mathrm{SL}_{#1}}}

\newcommand{\U}[1]{\ensuremath{\mathrm{U}_{#1}}}

\renewcommand{\det}{\ensuremath{\mathrm{det}}}

\newcommand{\HS}{\ensuremath{\mathbb{H}}}

\newcommand{\defeqr}{\mathrel{=:}}

\newcommand{\tdc}{\td{c}}

\newcommand{\Ga}{\Gamma}
\newcommand{\ga}{\gamma}
\newcommand{\la}{\lambda}
\newcommand{\om}{\omega}
\newcommand{\ka}{\kappa}

\newcommand{\tdGa}{\wtd{\Ga}}

\newcommand{\const}{\wtd{\mathrm{const}}}
\newcommand{\ord}{\mathrm{ord}}

\newcommand{\Heis}{\rmH}

\newcommand{\Ind}{\mathrm{Ind}}
\newcommand{\Res}{\mathrm{Res}}

\newcommand{\Hecke}[1]{\mathrm{Hecke}_{#1}}
\newcommand{\HeckeTor}[1]{\mathrm{HeckeTor}_{#1}}

\newcommand{\hecke}[1]{\mathrm{hecke}_{#1}}
\newcommand{\hecketor}[1]{\mathrm{hecketor}_{#1}}

\newcommand{\iotahecke}[1]{\iota_{\mathrm{Hecke},#1}}
\newcommand{\iotator}[1]{\iota_{\mathrm{Tor},#1}}

\newcommand{\lcm}{\mathrm{lcm}}

\newcommand{\OK}{\cO_K}
\newcommand{\OKell}{\cO_{K,\ell}}

\newcommand{\GMp}[1]{\mathrm{GMp}_{#1}}

\newcommand{\headertitle}{{%
  Explainable Ramanujan-type Congruences
}}
\newcommand{\headerauthors}{%
  M.~Raum%
}
\title{%
  Explainable Ramanujan-type Congruences on Square-Classes of Arithmetic Progressions
}
\dedication{In memory of Lynne Walling, her contribution to\\the field and our community}
\author{%
  Martin Raum%
  \thanks{The author was partially supported by Vetenskapsr\aa det Grants~2019-03551 and~2023-04217.}%
}
\ifdraft{\date{\today\ at\ \thistime}}{\date{}}

\begin{document}

\thispagestyle{scrplain}
\begingroup
\deffootnote[1em]{1.5em}{1em}{\thefootnotemark}
\maketitle
\endgroup

\begin{abstract}
\small
\noindent
{\tbf Abstract:}
While examples of Ramanujan-type congruences are amply available via their relation to Hecke operators, it remains unclear which of them should be considered of combinatorial origin and which of them are mere artifacts of the connection with modular forms. Ranks and generalized ranks have been proposed as a tool to discern this question. We formalize this idea as ``explainable Ramanujan-type congruences'' with reference to Jacobi forms with singularities at torsion points, and then associated with them subspaces in a specific complex representation of the modular group. The resulting representation theoretic perspective allows us to prove that explainable Ramanujan-type congruences occur on square-classes~$M \ZZ + u^2 \beta$ of arithmetic progressions.
\\[.3\baselineskip]
\noindent
\textsf{\textbf{%
  Jacobi forms with singularities%
}}%
\noindent
\ {\tiny$\blacksquare$}\ %
\textsf{\textbf{%
  vanishing loci%
}}%
\noindent
\ {\tiny$\blacksquare$}\ %
\textsf{\textbf{%
  colored partition function%
}}
\\[.2\baselineskip]
\noindent
\textsf{\textbf{%
  MSC Primary:
  11F30%
}}%
\ {\tiny$\blacksquare$}\ %
\textsf{\textbf{%
  MSC Secondary:
  11F33, 11F50, 11P83, 05A17%
}}
\end{abstract}

\vspace{-2.5\baselineskip}
\renewcommand{\contentsname}{}
\setcounter{tocdepth}{\sectiontocdepth}
\tableofcontents
\setcounter{tocdepth}{\subsectiontocdepth}
\vspace{1.5\baselineskip}

\Needspace*{4em}
\phantomsection
\label{sec:introduction}
\addcontentsline{toc}{section}{Introduction}
\markright{Introduction}
\lettrine[lines=2,nindent=.2em]{\tbf T}{\,he} integer partition function~$p(n)$ records the number of ways a number~$n$ can be written as a sum over a decreasing sequence of non-negative integers. Ramanujan congruences for~$p(n)$ were originally observe by Ramanujan as divisibility patterns~\cite{ramanujan-1920}. For instance, $5$ divides~$p(5 n + 4)$ for all~$n \in \ZZ$, and similar divisibilities by~$7$ and~$11$ hold on the arithmetic progressions~$7 \ZZ + 5$ and~$11 \ZZ + 6$.

These congruences not only are among the historical foundations of the field of combinatorial aspects of partitions of integers (MSC 05A17), but they are distinguished among many of their generalizations. This distinction is related to two different perspective taken on them in modern mathematics. First, it is fruitful to view them literally as divisibility patterns and asks for how often similar patterns occur. This leads to the notions of Ramanujan-like, e.g.\@ $17 \isdiv p_3(17 n + 15)$ for the~$3$\nbd{}colored partition function~$p_3$~\cite{boylan-2004}, and Ramanujan-type congruences, e.g.\@ $13 \isdiv p(59^4\, 13 n + 111247)$~\cite{atkin-obrien-1967}. Since the early~2000s when Ono connected Ramanujan-type congruences to the Hecke theory of modular forms~\cite{ono-2000,ahlgren-ono-2001,treneer-2006}, there is ample supply of them.

Second, one may require a ``combinatorial origin'' of Ramanujan-type congruence as evidence for a deeper combinatorial symmetry on, say, the set of integer partitions. The ranks of partitions provide a combinatorial explanation for the Ramanujan congruences modulo~$5$ and~$7$\cite{dyson-1944,atkin-swinnerton-dyer-1954}. Similarly, the congruences modulo~$11$ are explained by a generalized rank due to Dyson\footnote{Dyson called his invariants ``crank'', a then pun which seems historically outdated in light of the modern meaning of the word.}\cite{dyson-1944,andrews-garvan-1988}. At the foundation we have the question: Which Ramanujan-type congruences are genuinely combinatorial and which are artifacts of the theory of modular forms? Currently, we have combinatorial explanations for only very few of the Ramanujan-type congruences constructed through Ono's machinery.

The present work was triggered by Stanton, who in one of his talks advertised a formalization of generalized ranks. Given sets~$A(n)$, $n \in \ZZ$, of combinatorial interest---say the sets of specific restricted partitions of~$n$ or of Durfee symbols~\cite{andrews-2007}---that satisfy a Ramanujan-type congrunce modulo a prime~$\ell$ on the arithmetic progression~$M \ZZ + \beta$, i.e.\@ $\ell \isdiv \# A(M n + \beta)$ for all~$n \in \ZZ$, is there a generalized rank function
\begin{multline*}
  \rk(a) \defcol
  \bigcup_{n \in \ZZ_{\ge 0}} A(n)
\lra
  \ZZ
\quad\tx{such that }
\\
  \tx{for fixed~$n \in M \ZZ + \beta$}\quad
  \big\{
  a \in A(n)
  \condcol
  \rk(a) \equiv r \,\pmod{\ell}
  \big\}
  \tx{\ has equal size for all~$r \in \ZZ$?}
\end{multline*}
In Definition~\ref{def:explainable_ramanujan_type_congruence} we introduce the notion of \emph{explainable Ramanujan-type congruences} of weakly holomorphic modular forms, which combines Stanton's criterion with an additional requirement on modularity.

To justify the modularity requirement that we impose, we recall that many known instances of generalized ranks, although not all of them~\cite{bringmann-lovejoy-osburn-2010,andrews-chan-kim-osburn-2016} (see also Remark~\ref{rm:generalization_to_further_singularities}), can be assembled as a generating series
\begin{gather*}
  \sum_{\substack{n \in \ZZ\\a \in A(n)}}
  q^n \zeta^{\rk(a)}
\end{gather*}
for a~priori formal variables~$q$ and~$\zeta$ that affords an interpretation in terms of the (shifted) Fourier expansion of a Jacobi form. The rank and Dyson's generalized rank of integer partitions are examples, and there are further, more recent examples by Rolen-Tripp-Wagner for colored partitions~\cite{rolen-tripp-wagner-2020-preprint}.

The goal of the present paper is to initiate a theory of explainable Ramanujan-type congruences that leverages the modularity assumption to restrict their existence. We extend one of the foundational structure theorems for Ramanujan-type congruences of weakly holomorphic modular forms~\cite{radu-2012,radu-2013,raum-2022,raum-2023}.

To state our main theorem and its corollary, we let~$\chi_\theta$ and~$\chi_\eta$ be the characters of~$\tdGa_0(4) \subseteq \Mp{1}(\ZZ)$ and~$\Mp{1}(\ZZ)$ associated with the Jacobi theta function and the Dedekind eta function. For clarity, we mentioned that modular forms for~$\Mp{1}(\ZZ)$ and its subgroups historically were prior to the work of Kubota~\cite{kubota-1969}, and often still are, formulated via corresponding multiplier systems. We remind the reader of the notion of weakly holomorphic Jacobi forms with possible singularities at~$\ZZ + \tau \ZZ$ (see Section~\ref{sec:jacobi_forms}). Further, for a rational number~$x$ we let~$\ord_p(x)$ be the maximal integer~$n$ such that~$p^n \isdiv x$.

\begin{maintheorem}
\label{mainthm:square_classes_erc}
Let~$f$ be a weakly holomorphic modular form and~$\phi$ a weakly holomorphic Jacobi form with possible singularities at~$\ZZ + \tau \ZZ$ for~$\tdGa_0(N) \subseteq \Mp{1}(\ZZ)$ and character~$\chi \chi_\theta$ or~$\chi \chi_\eta$, where~$\chi$ is a Dirichlet character modulo~$N$. If~$f$ has a Ramanujan-type congrunce modulo a prime~$\ell$ on~$M \ZZ + \beta$ explained by~$\phi$ in the sense of Definition~\ref{def:explainable_ramanujan_type_congruence}, then~$f$ has explainable Ramanujan-type congruences modulo~$\ell$ on~$M \ZZ + u^2 \beta$ for all~$u \in \ZZ$ with~$\gcd(M, u) = 1$.
\end{maintheorem}

\begin{mainremark}
The result that we proof in Theorem~\ref{thm:square_classes_erc} applies to a much broader class of explainable Ramanujan-type congrunces, but is more technical to state.
\end{mainremark}

The next statement is an important application of Theorem~~\ref{mainthm:square_classes_erc}. Its analogue in the case of Ramanujan-type congruences enabled many of the arguments in, for instance, the second half of our work Ramanujan-type congruences in integral weight~\cite{raum-2022}, and in the case of Hurwitz class numbers~\cite{beckwith-raum-richter-2022} is crucial for later work on~$\ell$\nbd{}torsion in class groups~\cite{beckwith-raum-richter-2023-preprint}. We intend to leverage Corollary~\ref{maincor:max_erc} in future work, whose context we mention at the end of this introduction.

\begin{maincorollary}
\label{maincor:max_erc}
Let~$f$ be a weakly holomorphic modular form. Given a prime~$\ell$, consider the set of arithmetic progressions that support explainable Ramanujan-type congruences modulo~$\ell$:
\begin{multline*}
  \big\{
  M \ZZ + \beta
  \condcol
  f \tx{\ has a Ramanujan-type congrunce modulo~$\ell$ on~$M \ZZ + \beta$}
\\
  \tx{explained by some~$\phi$ as in Theorem~\ref{mainthm:square_classes_erc}}
  \big\}
\tx{.} 
\end{multline*}
Its maximal elements~$M \ZZ + \beta$ with respect to inclusion satisfy
\begin{alignat*}{2}
  \ord_p(M) &\le \max\{0, \ord_p(\beta) + 1 \}
\tx{,}\quad
&&
  \tx{if\ } p \ne 2 \tx{\ prime}
\tx{;}
\\
  \ord_2(M) &\le \max\{0, \ord_2(\beta) + 3 \}
\tx{,}\quad
&&
  \tx{if\ } p = 2 \tx{\ prime}
\tx{.}
\end{alignat*}
\end{maincorollary}

The proofs of Theorem~\ref{mainthm:square_classes_erc} and its Corollary~\ref{maincor:max_erc} appear in Section~\ref{ssec:representation_theory:square_classes}. Both statements are consequences Theorem~\ref{thm:square_classes_erc}.

Our approach to explainable Ramanujan-type congruences is inspired by previous work on relations among Ramanujan-type congruences, which centers around the representation theory of~$\SL{2}(\ZZ)$ and its metaplectic cover~$\Mp{1}(\ZZ)$. Our emphasis in the present work lies on a unified point of view on Ramanujan-type congruences and explainable Ra\-ma\-nu\-jan-type congruences. While in previous work, we characterized Ramanujan-type congruences employing modular representation theory, i.e.\@ representation theory over rings of finite characteristic, explainable Ramanujan-type congruences lead to complex representations. Comparing the relevant representations that appear in Lemma~\ref{la:hecke_representation} and Proposition~\ref{prop:hecke_torsion_representation}, it becomes evident why similar methods should apply to both situations. We substantiate this impression in the present paper.

It is natural to ask to what extend the semi-simple representation theory in characteristic~$0$ that governs explainable Ramanujan-type congruences can provide insight beyond what we know about Ramanujan-type congruences. Preliminary considerations for~$\SL{2}(\ZZ)$ based on Corollary~\ref{maincor:max_erc} suggest that explainable Ramanujan-type are significantly more restricted than mere Ramanujan-type congruences. We plan to address this and the more delicate case of~$\Mp{1}(\ZZ)$ in a sequel.

\section{Ramanujan-type congruences}%
\label{sec:ramanujan_type}

This section is based on previous work on Ra\-ma\-nu\-jan-type congruences~\cite{raum-2022,raum-2023}, but provides a new perspective that we extend in the subsequent sections to explainable Ra\-ma\-nu\-jan-type congruences. The proof of Theorem~\ref{thm:square_classes_rc} and its Corollary~\ref{cor:square_classes_rc_eta_theta}, whose statement already appears in those papers, differs from the previously given argument. While in previous work we performed explicit calculations of character values, we develop a more flexible method, which puts more emphasis on the interpretation of characters as one-dimensional representations on which we have a conjugation action by a suitable group defined in~\eqref{eq:def:gamma_diag_char}.

Throughout, $e(z) = \exp(2 \pi i\, z)$ denotes the renormalized exponential. We write~$\HS = \{ \tau \in \CC \condcol \Im(\tau) > 0 \}$ for the Poincar\'e upper half plane, which carries an action of~$\SL{2}(\RR)$ by M\"obius transformations. We realize the metaplectic cover~$\Mp{1}(\RR)$ of~$\SL{2}(\RR)$ as a set of pairs~$(\ga,\om)$, where~$\ga = \begin{psmatrix} a & b \\ c & d \end{psmatrix} \in \SL{2}(\RR)$ and~$\om$ is a holomorphic square-root of~$c \tau + d$ on~$\HS$. Its multiplication is then given by~$(\ga_1, \om_1) (\ga_2, \om_2) = (\ga_1 \ga_2,\, \om_1 \circ \ga_2 \cdot \om_2)$.

We write~$\Mp{1}(\ZZ) \subset \Mp{1}(\RR)$ for the preimage of~$\SL{2}(\ZZ)$ under the projection to~$\SL{2}(\RR)$, and~$\tdGa_0(N) \thra \Ga_0(N)$, $\tdGa^0(N) \thra \Ga^0(N)$, $\tdGa(N) \thra \Ga(N)$ for the corresponding level\nbd{}$N$ congruences subgroups. We write~$T$ and~$S$ for the generators~$\begin{psmatrix} 1 & 1 \\ 0 & 1 \end{psmatrix}$ and~$\begin{psmatrix} 0 & -1 \\ 1 & 0 \end{psmatrix}$ of~$\SL{2}(\ZZ)$, and identify them with their lifts $T \defeq \big( \begin{psmatrix} 1 & 1 \\ 0 & 1 \end{psmatrix},\, 1 \big)$ and $S \defeq \big( \begin{psmatrix} 0 & -1 \\ 1 & 0 \end{psmatrix},\, \sqrt{\tau} \big)$ to~$\Mp{1}(\ZZ)$, where the square root is the principle branch.

Via the natural projection~$\Mp{1}(\RR) \thra \SL{2}(\RR)$ to the special linear group the metaplectic group acts on~$\HS$. The action on~$\HS$ lifts to a family of actions on functions on~$\HS$ parametrized by weights~$k \in \frac{1}{2} \ZZ$:
\begin{gather*}
  \big( f \big|_k\, (\ga,\om) \big)(\tau)
\defeq
  \om(\tau)^{2k}\,
  f\big(\mfrac{a \tau + b}{c \tau + d}\big)
\tx{,}\quad
  \ga = \begin{psmatrix} a & b \\ c & d \end{psmatrix}
\tx{.}
\end{gather*}

A weakly holomorphic modular form for a finite index subgroup~$\Ga \subseteq \Mp{1}(\ZZ)$, character~$\chi$ of~$\Ga$, and weight~$k \in \frac{1}{2} \ZZ$ is a holomorphic function~$f \defcol \HS \ra \CC$ that satisfies
\begin{enumerateroman}
\item $f |_k\, (\ga,\om) = \chi(\ga,\om)\, f$ for all~$(\ga,\om) \in \Ga$,
\item there is~$a \in \RR$ with $|(f |_k (\ga,\om) )(\tau)| \ll \exp(a \Im(\tau))$ as~$\Im(\tau) \ra \infty$ for all~$(\ga,\om) \in \Mp{1}(\ZZ)$.
\end{enumerateroman}
The space of such functions is written as~$\rmM^!_k(\Ga,\chi)$. A weakly holomorphic modular form~$f$ admits a Fourier expansion
\begin{gather*}
  \sum_{n \in \QQ}
  c(f; n) e(n \tau)
\tx{,}
\end{gather*}
where~$c(f; n) = 0$ if~$n$ sufficiently small. If~$\chi$ has level~$N$, then the Fourier coefficients of~$f$ is supported on~$\frac{1}{N} \ZZ$.

We can view the~$\CC$\nbd{}vector space~$\rmM^!_k(\Ga,\chi)$ as a right-representation of~$\Ga$, which acts via the slash action of weight~$k$. Further, since~$\chi$ is a character of~$\Ga$, it yields a one-dimensional representation, which we identify with it. By modular invariance imposed in the definition of weakly holomorphic modular forms, $\rmM^!_k(\Ga,\chi)$ is~$\chi$\nbd{}isotypical, that is, as a representation of~$\Ga$ it is isomorphic to a direct sum of copies of~$\chi$. To contextualize this, we note that if~$\Ga \subset \Mp{1}(\ZZ)$ is normal, then~$\rmM^!_k(\Ga,\chi)$ also carries an action of~$\Mp{1}(\ZZ)$ and any quotient-representation of the induction of~$\chi$ to~$\Mp{1}(\ZZ)$ can a priori occur.

Given a right-representation~$V$ of a group~$H \subseteq G$ over a ring~$R$, we view it as a module over the group algebra~$R[H]$ and define the induction via the~$R[G]$\nbd{}module
\begin{gather*}
  \Ind_H^G\,V
\defeq
  V \otimes_{R[H]} R[G]
\tx{.}
\end{gather*}
Paralleling this notation, the restriction of a representation~$V$ of~$G$ to~$H$ is denoted~$\Res_H^G\,V$.

In order to  make~$\rmM^!_k(\Ga,\chi)$ amenable to arguments from modular representation theory, in~\cite{raum-2022} we introduced the notion of abstract spaces of weakly holomorphic modular forms. They are defined as pairs~$(V,\psi)$ for a morphism~$\psi$ (denoted~$\phi$ in previous work, which we will reserve for Jacobi forms in this paper) from a representation~$V$ of a finite index subgroup~$\Ga \subseteq \Mp{1}(\ZZ)$ to~$\rmM^!_k(\Ga')$ for some potentially different group~$\Ga' \subseteq \Ga$, the kernel of~$V$. We will skip this part of the theory in favor of a more explicit handling of the homomorphism~$\psi$ and its range. This, we hope, provides not only an alternative perspective but also separates more strictly concerns from representation theory and modular forms.

The starting point of~\cite{raum-2022} was the observation that the so-called Fourier expansion principle can be understood in representation theoretic context. Such ideas in different disguise have long been know and used in the geometric perspective on modular forms pioneered by Katz~\cite{katz-1973} and refined by Deligne-Rapoport~\cite{deligne-rapoport-1973}. To accommodate the case of half-integral weights, we required unpublished work by Jochnowitz~\cite{jochnowitz-2004-preprint} as well.

Throughout this section, we fix a number field~$K \subset \CC$, and let~$\OKell$ be the localization of the ring of integers~$\OK \subset K$ at a prime ideal~$\ell$. Given a $\CC$\nbd{}vector space~$V$ of weakly holomorphic modular forms, we will commonly write
\begin{gather}
\label{eq:def:integral_fourier_coefficients}
  \bbV
\defeq
  \big\{
  f \in V
  \condcol
  \forall n \in \QQ
  \quantsep
  c(f;n) \in \OKell
  \big\}
\end{gather}
for the $\OKell$\nbd{}submodule of all its elements with Fourier coefficients in~$\OKell$.

The next statement corresponds to Theorem~2.2 of~\cite{raum-2022}, but does not pinpoint the level~$N$.

\begin{theorem}
\label{thm:mf_representation}
Let~$V$ be a $\CC$\nbd{}vector space of weakly holomorphic modular forms of given weight~$k \in \frac{1}{2}\ZZ$ for a congruence subgroup. Assume that~$V$ is preserved by the action of\/~$\Ga \subseteq \Mp{1}(\ZZ)$. Then~$\bbV$ defined as in~\eqref{eq:def:integral_fourier_coefficients} is preserved by the action of\/~$\Ga \cap \tdGa_0(N)$ for a suitable positive integer~$N$.
\end{theorem}

The restriction to~$\Ga \cap \tdGa_0(N_\ell)$ that appears in the theorem originates directly from Deligne-Rapoport and cannot be removed as it inherently tied to the geometry of the moduli of elliptic curves over characteristic~$\ell$. Note that as~$\bbV$ is preserved by the action of~$\Ga \cap \tdGa_0(N_\ell)$, then so is~$\ell \bbV$, which is the submodule of forms~$f \in V$ whose coefficients are divisible by~$\ell$, i.e.\@ contained in the ideal~$\ell$.

The second insight that enabled a representation theoretic perspective on Ramanujan-type congruences in previous work is the observation that abstract spaces of modular forms encompass the dual of vector-valued Hecke operators~\cite{raum-2017}. It is an idea inspired by other fields, namely the local theory of automorphic representations where fixed vectors for compact open subgroups of, say, $\SL{2}(\QQ_p)$ play a crucial role. We replace the formulation used in~\cite{raum-2022} by the setup provided in Lemma~\ref{la:hecke_representation}. Note that it combines the vector-valued Hecke operators from~\cite{raum-2022} and induction to~$\Mp{1}(\ZZ)$ in one statement.

We need the extension~$\GMp{1}^+(\RR)$ of~$\Mp{1}(\RR)$ to positive similitudes, the usual extension of the slash action to~$\GMp{1}^+(\RR)$ normalized by the determinant
\begin{gather*}
  \big( f \big|_k (\ga, \om) \big)(\tau)
\defeq
  \det(\ga)^{\frac{k}{2}}\,
  \om(\tau)^{-2k}\,
  f(\ga \tau)
\tx{,}  
\end{gather*}
and for positive integers~$M$ the sets
\begin{gather*}
  \GL{2}^{(M)}(\ZZ)
\defeq
  \big\{
  \ga \in \Mat{2}(\ZZ)
  \condcol
  \det(\ga) = M
  \big\}
\tx{,}\quad
  \GMp{1}^{(M)}(\ZZ)
\defeq
  \big\{
  (\ga,\om) \in \GMp{1}(\RR)
  \condcol
  \ga \in \GL{2}^{(M)}(\ZZ)
  \big\}
\tx{.}
\end{gather*}
We will use the element~$\delta_M \defeq \begin{psmatrix} 1 & 0 \\ 0 & M \end{psmatrix}$ and identify it with its lift~$\big( \begin{psmatrix} 1 & 0 \\ 0 & M \end{psmatrix}, \sqrt{M} \big)$. It allows us to define the representation
\begin{gather}
\label{eq:def:hecke_representation}
  \Hecke{M}(V)
\defeq
  \Ind_{\delta_M^{-1} \Ga \delta_M \cap \Mp{1}(\ZZ)}^{\Mp{1}(\ZZ)}\, \iota^\ast(V)
\end{gather}
for a representation~$V$ of~$\Ga$. The most common case will be~$V = \chi$ for a character~$\chi$ of~$\Ga \subseteq \Mp{1}(\ZZ)$. Recall that this, for instance, applies if~$V = \CC f$ for a weakly holomorphic modular form for~$\chi$.

\begin{lemma}%
\label{la:hecke_representation}
Suppose that~$V$ is a space of functions on~$\HS$ that is stable under the weight~$k$ slash action of a subgroup~$\Ga \subseteq \Mp{1}(\ZZ)$. Then given a positive integer~$M$ we have the homomorphism of\/~$\Ga$\nbd{}representations
\begin{gather}
\label{la:eq:hecke_representation}
\begin{aligned}
  \Hecke{M}(V)
&\lra
  \linspan_\CC\big\{
  f \big|_k\, \delta
  \condcol
  f \in V,\ 
  \delta \in \GMp{1}^{(M)}(\ZZ)
  \big\}
\tx{,}
\\
  f \otimes (\ga,\om)
&\lmto
  f \big|_k\, \delta_M (\ga, \om)
\tx{,}
\end{aligned} 
\end{gather}
where~$\iota^\ast(V)$ is the pullback of\/~$V$ along the embedding
\begin{gather*}
  \iota
=
  \iotahecke{M}
\defcol
  \tdGa^0(M)
\lra
  \Mp{1}(\ZZ)
\tx{,}\quad
  \big( \begin{psmatrix} a & b \\ c & d \end{psmatrix}, \om \big)
\lmto
  \big( \begin{psmatrix} a & b \slash M \\ M c & d \end{psmatrix}, \tau \mto \om(M \tau) \big)
\tx{.}
\end{gather*}
\end{lemma}

\begin{proof}
Since~$\delta_M \ga$ has determinant~$M$ for all~$\ga \in \SL{2}(\ZZ)$, the image given in~\eqref{la:eq:hecke_representation} lies in the stated co-domain. To see that the map is well-defined and yields a homomorphism, il suffices to verify that for~$(\ga,\om) \in \delta_M^{-1} \Ga \delta_M \cap \Mp{1}(\ZZ)$ we have
\begin{gather*}
  \Big(
  f
  \big|_k \begin{psmatrix} 1 & 0 \\ 0 & M \end{psmatrix}
  \Big)
  \Big|_k (\ga, \om)
=
  \Big( f
  \big|_k \big( \begin{psmatrix} a & b \slash M \\ M c & d \end{psmatrix}, \tau \mto \om(M \tau) \big)
  \Big)
  \Big|_k \begin{psmatrix} 1 & 0 \\ 0 & M \end{psmatrix}
\tx{.}
\end{gather*}
\end{proof}

Given a weakly holomorphic modular form~$f$ whose Fourier coefficients are supported on~$\beta + \ZZ$ for some~$\beta \in \QQ$, we find that
\begin{gather}
\label{eq:fourier_expansion_projection_rc}
  M^{1 - \frac{k}{2}}\,
  \sum_{h \pmod{M}}
  e\big( \mfrac{-h \beta}{M} \big)\,
  f \big|_k\, 
  \big( \begin{psmatrix} 1 & h \\ 0 & M \end{psmatrix}, \sqrt{M} \big)
=
  \sum_{n \in \beta + M \ZZ}
  c(f;n)\,
  e\big( \mfrac{\tau}{M} \big)
\tx{.}
\end{gather}
Note that the summation over~$h \pmod{M}$ is well-defined because we have~$f |_k\, T = e(\beta) f$. The role played by this formula corresponds to Proposition~2.4 in~\cite{raum-2022}. However, it does not include a representation theoretic perspective yet, which we formulate as the next corollary to Lemma~\ref{la:hecke_representation}. Since~$\CC f$ is~$\chi$\nbd{}isotypical if~$f$ is a weakly holomorphic modular form for~$\chi$, this corollary allows us to examine the Fourier expansions in~\eqref{eq:fourier_expansion_projection_rc} without reference to the weakly holomorphic modular form~$f$, but only its character. We need the element
\begin{gather}
\label{eq:def:hecke_vector}
  \hecke{M}(\chi; \beta)
\defeq
  M^{1 - \frac{k}{2}}\mspace{-12mu}
  \sum_{h \pmod{M}}
  e\big( \mfrac{-h \beta}{M} \big)\, T^h
\mathrel{\;\in\;}
  \Hecke{M}(\chi)
\tx{.}
\end{gather}

Given a weakly holomorphic modular form for a character of~$\Ga \subset \Mp{1}(\ZZ)$ with Fourier expansion supported on~$\beta + \ZZ$, we will always assume without loss of generality that~$T \in \Ga$ and~$\chi(T) = e(\beta)$. Note however that this does not mean that~$\hecke{M}(\chi; \beta)$ equals a multiple of the identity in~$\Mp{1}(\ZZ)$, since it lies in the induction of a pullback of~$\chi$.

The next statement parallels~Corollary~2.5 of~\cite{raum-2022}.

\begin{corollary}
\label{cor:rc_characterization}
Fix a weakly holomorphic modular form~$f$ for a character~$\chi$ of\/~$\Ga \subseteq \Mp{1}(\ZZ)$. Assume that the Fourier coefficients of~$f$ are supported on~$\beta + \ZZ$ for some~$\beta \in \QQ$. Then given a positive integer~$M$, $f$ has a Ramanujan-type congruence modulo~$\ell$ on~$M \ZZ + \beta$ if and only if\/~$\psi(\hecke{M}(\chi;\beta)) \in \ell \bbV$, where~$\psi$ is the homomorphism in Lemma~\ref{la:hecke_representation} and~$V = \CC f$.
\end{corollary}

We merely want to reprove a small part of the results of~\cite{raum-2022,raum-2023}, and for this it suffices to examine subrepresentations of the induction in Lemma~\ref{la:hecke_representation}. Note that the next lemma has no analogue in~\cite{raum-2022}, but appears as part of the proof of Proposition~3.5 in loc.\@ cit. We set
\begin{gather}
\label{eq:def:gamma_diag}
  \Ga^\Delta_M
\defeq
  \big\{
  \ga = \begin{psmatrix} a & b \\ c & d \end{psmatrix}
  \in \Ga
  \condcol
  T^{a^2 h} \ga T^{-h} \in \delta_M^{-1} \Ga \delta_M
  \tx{\ for all\ } h \in \ZZ,\ 
  \ga \tx{\ diagonal modulo\ }M
  \big\}
\end{gather}
and
\begin{gather}
\label{eq:def:gamma_diag_char}
  \Ga^{\Delta}_M(\chi)
\defeq
  \big\{
  \ga = \begin{psmatrix} a & b \\ c & d \end{psmatrix} \in \Ga^\Delta_M
  \condcol
  \wtd\chi\big( T^{a^2 h} \ga T^{-h} \big) = \wtd\chi(\ga)
  \tx{\ for all\ } h \in \ZZ,\ 
  \big\}
\quad\tx{with\ }
  \wtd\chi = \iotahecke{M}^\ast\,\chi
\tx{.}
\end{gather}

\begin{proposition}%
\label{prop:action_on_hecke_vector_span}
Fix a group~$\Ga \subseteq \Mp{1}(\ZZ)$, a character~$\chi$ of\/~$\Ga$, a positive integer~$M$, and~$\beta \in \QQ$. Then~$\Ga^\Delta_M$ is a group and acts on the subspace
\begin{gather*}
  \linspan_\CC
  \big\{
  \hecke{M}(\chi; \beta)
  \big\}
\subseteq
  \Hecke{M}(\chi)
\tx{.}
\end{gather*}
Further, $\Ga^\Delta_M(\chi)$ is a group that acts projectively on the vectors~$\hecke{M}(\chi; \beta)$ with character~$\wtd\chi$ as in~\eqref{eq:def:gamma_diag_char}:
\begin{gather*}
  \hecke{M}(\chi; \beta)\, \ga
=
  \wtd\chi(\ga)\,
  \hecke{M}(\chi; a^2 \beta)
\tx{,}\quad
  \ga
=
  \begin{psmatrix} a & b \\ c & d \end{psmatrix}
\in
  \Ga^\Delta_M(\chi)
\tx{.}
\end{gather*}
\end{proposition}

\begin{proof}
To see that~$\Ga^\Delta_M$ and~$\Ga^\Delta_M(\chi)$ are groups, we consider
\begin{gather*}
  \ga
=
  \begin{psmatrix} a & b \\ c & d \end{psmatrix}
\equiv
  \begin{psmatrix} a & 0 \\ 0 & d \end{psmatrix}
  \;\pmod{M}
\quad\tx{and}\quad
  \ga'
=
  \begin{psmatrix} a' & b' \\ c' & d' \end{psmatrix}
\equiv
  \begin{psmatrix} a' & 0 \\ 0 & d' \end{psmatrix}
  \;\pmod{M}
\tx{,}
\end{gather*}
and note that
\begin{gather*}
  T^{a^2 a^{\prime 2} h} \ga T^{-a^{\prime 2} h}
  T^{a^{\prime 2} h} \ga' T^{-h}
=
  T^{a^2 a^{\prime 2} h} 
  \ga \ga'
  T^{-h}
\quad\tx{and}\quad
  \ga \ga'
\equiv
  \begin{psmatrix} a a' & 0 \\ 0 & d d' \end{psmatrix}
  \;\pmod{M}
\tx{.}
\end{gather*}

If~$\ga \in \Ga^\Delta_M$, we have
\begin{gather*}
  T^h \ga
=
  \big( T^h \ga T^{-d^2 h} \big)\,
  T^{d^2 h}
\quad\tx{with\ }
  T^h \ga T^{-d^2 h}
=
  T^{(1-a^2 d^2) h}\,
  T^{a^2 d^2 h} \ga T^{-d^2 h}
\in
  \delta_M^{-1} \Ga \delta_M \cap \Mp{1}(\ZZ)
\tx{,}
\end{gather*}
since~$M$ divides~$1 - a^2 d^2$ by definition of~$\Ga^\Delta_M$. In conjunction with the definition of the representation~$\Hecke{M}(\chi)$ in~\eqref{eq:def:hecke_representation} this implies that~$\Ga^\Delta_M$ acts as stated.

To show that~$\ga \in \Ga^\Delta_M(\chi)$ acts as claimed, we consider
\begin{align*}
&
  \hecke{M}(\chi;\beta)\, \ga
=
  M^{1 - \frac{k}{2}}\mspace{-12mu}
  \sum_{h \pmod{M}}
  e\big( \mfrac{-h \beta}{M} \big)\, T^h \ga
=
  M^{1 - \frac{k}{2}}\mspace{-12mu}
  \sum_{h \pmod{M}}
  e\big( \mfrac{-h a^2 \beta}{M} \big)\, T^{a^2 h} \ga
\\
={}&
  M^{1 - \frac{k}{2}}\mspace{-12mu}
  \sum_{h \pmod{M}}
  e\big( \mfrac{-h a^2 \beta}{M} \big)\, T^{a^2 h} \ga T^{-h}\, T^h
=
  M^{1 - \frac{k}{2}}\mspace{-12mu}
  \sum_{h \pmod{M}}
  e\big( \mfrac{-h a^2 \beta}{M} \big)\, \wtd\chi\big( T^{a^2 h} \ga T^{-h} \big)\, T^h
\\
={}&
  \wtd\chi(\ga)\,
  M^{1 - \frac{k}{2}}\mspace{-12mu}
  \sum_{h \pmod{M}}
  e\big( \mfrac{-h a^2 \beta}{M} \big)\, T^h
\\
={}&
  \wtd\chi(\ga)\,
  \hecke{M}(\chi, a^2 \beta)
\tx{.}
\end{align*}
\end{proof}

It remains to combine Proposition~\ref{prop:action_on_hecke_vector_span} with Theorem~\ref{thm:mf_representation} in order to obtain the relation among Ramanujan-type congurences on square-classes of~$\beta \,\pmod{M}$.

\begin{theorem}
\label{thm:square_classes_rc}
Let~$f$ be a weakly holomorphic modular form for a character~$\chi$ of\/~$\Ga \subseteq \Mp{1}(\ZZ)$. Assume that the kernel of~$\chi$ is a congruence subgroup.

If~$f$ has a Ramanujan-type congruence modulo~$\ell$ on~$M \ZZ + \beta$, then it has such congruences on~$M \ZZ + a^2 \beta$ for all~$a \in \ZZ$ such that there is~$\begin{psmatrix} a & b \\ c & d \end{psmatrix} \in \Ga^{\prime\Delta}_M(\chi)$, where~$\Ga' = \Ga \cap \Ga_0(N)$ for~$N$ as in Theorem~\ref{thm:mf_representation}.
\end{theorem}

Before we prove Theorem~\ref{thm:square_classes_rc}, we specialize its last condition to the case of the characters~$\chi_\theta$ and~$\chi_\eta$ of the Jacobi~$\theta$ function and the Dedekind~$\eta$ function, we obtain the following statement that recovers Proposition~3.5 of~\cite{raum-2022} and Proposition~3.1 of~\cite{raum-2023}.

\begin{corollary}
\label{cor:square_classes_rc_eta_theta}
Let~$f$ be a weakly holomorphic modular form for the character~$\chi \chi_\eta$ or~$\chi \chi_\theta$ of\/~$\tdGa_0(N)$, where~$\chi$ is a Dirichlet character. If~$f$ has a Ramanujan-type congruence modulo~$\ell$ on~$M \ZZ + \beta$, then it has such congruences on~$M \ZZ + u^2 \beta$ for all~$u \in \ZZ$ with~$\gcd(M, u) = 1$.
\end{corollary}

\begin{proof}
Without loss of generality, we can assume that~$N$ is such that the conclusion of Theorem~\ref{thm:mf_representation} holds. Now it suffices to notice that for~$\Ga = \tdGa_0(N)$ and~$\chi$ as given, we have
\begin{gather*}
  \tdGa_0(24 M N)
\subseteq
  \Ga^{\prime\Delta}_M(\chi)
\tx{.}
\end{gather*}
\end{proof}

\begin{proof}[Proof of Theorem~\ref{thm:square_classes_rc}]
We let~$V = \CC f \cong \chi$ and~$k$ the weight of~$f$. Consider the homomorphism~$\psi$ from~$\Hecke{M}(\chi)$ to the span of all~$f |_k \delta$ for~$\delta \in \GMp{1}^{(M)}(\ZZ)$. Theorem~\ref{thm:mf_representation} asserts that~$\ell \bbV$ is a representation for~$\Ga'$. Hence so is~$\psi^{-1}(\ell \bbV)$.

Corollary~\ref{cor:rc_characterization} implies that we have~$\hecke{M}(\chi, \beta) \in \psi^{-1}(\ell \bbV)$. Then Proposition~\ref{prop:action_on_hecke_vector_span} shows that we have a unitary projective action of~$\Ga^{\prime\Delta}_M(\chi)$ on the set of all~$\hecke{M}(\chi, a^2 \beta)$, which implies the statement.
\end{proof}

\section{Jacobi forms and Ramanujan-type congruences}%
\label{sec:jacobi_forms}

We now turn to the role that Jacobi forms play in the topic of Ramanujan-type congruences. Our goal is to define a notion of explainable Ramanujan-type congruences based on generalized ranks. This notion formalizes the observations made by Rolen-Tripp-Wagner~\cite{rolen-tripp-wagner-2020-preprint}, who examined colored partition functions.

The general reference on Jacobi forms is the book by Eichler-Zagier~\cite{eichler-zagier-1985}. We adopt and adjust their notation. The real metaplectic Jacobi group is defined as the semi-direct product~$\Mp{1}(\RR) \ltimes \rmH(\RR)$ where~$\rmH(\RR)$ is the ``Heisenberg group'' that appears~\cite{eichler-zagier-1985}. We identify both elements of~$\Mp{1}(\RR)$ and~$\rmH(\RR)$, and even elements of~$\RR^2 \subset \rmH(\RR)$, with the corresponding elements of the real metaplectic Jacobi group. The Jacobi slash action of weight~$k \in \frac{1}{2}\ZZ$ and index~$m \in \CC$ on functions~$\phi \defcol \HS \times \CC \ra \CC$ is given by
\begin{align*}
  \big( \phi \big|_{k,m} (\ga, \om) \big)(\tau,z)
&=
  \omega(\tau)^{2k}\, e\big( m \mfrac{- c z^2}{c \tau + d} \big)\,
  \phi\big( \mfrac{a \tau + b}{c \tau + d},\, \mfrac{z}{c \tau + d} \big)
\tx{,}
\\[0.2\baselineskip]
  \big( \phi \big|_{k,m} (\la,\mu,\ka) \big)(\tau,z)
&=
  e\big( m (\la^2 \tau + 2 \la z + \kappa) \big)\,
  \phi\big( \tau,\, z + \la \tau + \mu \big)
\tx{.}
\end{align*}

Let~$\Ga \subseteq \Mp{1}(\ZZ)$ be a finite index subgroup. A meromorphic function~$\phi \defcol \HS \times \CC \ra \CC$ is a weakly holomorphic Jacobi form of weight~$k$ and index~$m$ for a character~$\chi$ of~$\Ga \ltimes \rmH(\ZZ)$ with singularities at torsion points, if
\begin{enumerateroman}
\item 
there is~$N \in \ZZ_{\ge 1}$ such that~$\phi$ regular outside of the locus~$\tau \in \frac{1}{N}(\ZZ + \tau \ZZ)$,
\item
it transforms like~$\chi$ under the Jacobi slash action of~$\Ga \ltimes \rmH(\ZZ)$, and
\item
for some~$a \in \RR$ satisfies~$|(\phi |_{k,m} (\ga,\om)) (\tau, z)| \ll \exp(a \Im(\tau))$ as~$\Im(\tau) \ra \infty$ for all~$(\ga,\om) \in \Mp{1}(\ZZ)$ and~$z \in \QQ + \ga \tau \QQ$ for which it is regular.
\end{enumerateroman}
We say that $\Ga \ltimes \Lambda \subseteq \Mp{1}(\ZZ) \ltimes \rmH(\ZZ)$ has level~$N$ if~$\Ga \subseteq \Mp{1}(\ZZ)$ has level~$N$ and~$\rmH(N \ZZ) \subseteq \Lambda$. 

Recall that a weakly holomorphic modular form~$f$ has a Ramanujan-type congruence modulo a prime~$\ell$ on~$M \ZZ + \beta$ if
\begin{gather*}
  \forall n \in M \ZZ + \beta
  \quantsep
  c(f; n) \equiv 0
  \;\pmod{\ell}
\tx{.}
\end{gather*}
Ramanujan-type congruences for which~$M = \ell$ are usually referred to as Ramanujan-like. The partitions examined by Rolen-Tripp-Wagner all instances of these.

Given a weakly holomorphic Jacobi form~$\phi$ with singularities at torsion points we say that~$\phi$ specializes (at~$z = 0$) to a weakly holomorphic modular form~$f$ if the leading term in the Laurent expansion for~$\phi$ around~$z = 0$ equals~$f$. As in the work of Rolen-Tripp-Wagner~\cite{rolen-tripp-wagner-2020-preprint}, we want to employ such~$\phi$ to explain Ramanujan-type congruences of~$f$. To achieve this, we need to assume that~$\phi$ has no further singularities in~$( \QQ + \tau \QQ ) \setminus (\ZZ + \tau \ZZ)$, that is, $\phi$ has singularities at most at~$\ZZ + \tau \ZZ$.

\begin{remark}
\label{rm:generalization_to_further_singularities}
It would be interesting to generalize at least parts of the theory that we present in this paper to cases when~$\phi$ has arbitrary singularities on torsion points. Such a generalization will usually lead to mock modular forms~\cite{dabholkar-murthy-zagier-2012-preprint}, whose Ramanujan-type congruences are not yet sufficiently well understood.
\end{remark}

Every weakly holomorphic Jacobi form~$\phi$ with singularities at torsion points admits a Fourier series expansion
\begin{gather}
\label{eq:def:fourier_expansion_jacobi}
  \phi(\tau, z)
=
  \sum_{n, r \in \QQ}
  c(\phi; n, r)\,
  e(n \tau + r z)
\end{gather}
that is valid for positive and sufficiently small~$\Im(z) \slash \Im(\tau)$.

We have~$c(\phi; n, r) = 0$ for sufficiently small~$n$, and if the character~$\chi$ of~$\phi$ has kernel of level~$N$, then the Fourier coefficients of~$\phi$ are supported on~$n, r \in \frac{1}{N}\ZZ$. If~$c(\phi;n,r) \in K$ for all~$n, r$, we say that is has~$\ell$\nbd{}integral Fourier coefficients for a prime ideal~$\ell$ if~$c(\phi; n, r) \in \OKell$. For notational convenience, we set
\begin{gather*}
  c(\phi; n; z)
\defeq
  \sum_{r \in \QQ}
  c(\phi; n, r)\,
  e(r z)
\tx{.}
\end{gather*}

Assume now that~$\phi$ has at most singularities at~$\ZZ + \tau \ZZ$, and let~$-\nu \in \ZZ$ be the order of~$\phi$ at~$z = 0$. Then we define its leading Fourier coefficients (with respect to~$\tau$) via
\begin{gather}
\label{eq:def:leading_fourier_coefficient}
  \tdc(\phi; n; z)\,
\defeq
  \big( e(\tfrac{1}{2} z) - e(- \tfrac{1}{2} z) \big)^\nu\,
  c(\phi; n; z)
\tx{,}
\end{gather}
that is, we have the expansion
\begin{gather*}
  \phi(\tau, z)
\defeqr
  \sum_{n \in \QQ}
  \big( e(\tfrac{1}{2} z) - e(- \tfrac{1}{2} z) \big)^{-\nu}\,
  \tdc(\phi; n; z)\,
  e(n \tau)
\tx{.}
\end{gather*}

Based on the following Proposition~\ref{prop:fourier_coefficients_specialization}, we will always consider~$\tdc(\phi; n; z)$ as a Laurent polynomial in~$e(z \slash N)$, where~$N$ is the level of the character of~$\phi$. In particular, any statement on divisibility of~$\tdc(\phi; n; z)$ will refer to divisibility in the ring~$\CC[e(z \slash N), e(-z)]$.

\begin{proposition}%
\label{prop:fourier_coefficients_specialization}
Suppose that~$\phi$ is as in~\eqref{eq:def:leading_fourier_coefficient}. Then the leading Fourier coefficients~$\tdc(\phi;n;z)$ are Laurent polynomials in~$e(z \slash N)$, where~$N$ is the level of the character of\/~$\phi$.
\end{proposition}

\begin{proof}
The function~$e(z \slash 2) - e(- z \slash 2)$ has a simple zero at~$z = 0$. We have the Jacobi theta function
\begin{gather*}
  \theta(\tau, z)
\defeq
  e(\tfrac{1}{8} \tau)\,
  \big( e(\tfrac{1}{2} z) - e(-\tfrac{1}{2} z) \big)
  \prod_{n = 1}
  \big( 1 - e(n \tau) \big)
  \big( 1 - e(n \tau + z) \big)
  \big( 1 - e(n \tau - z) \big)
\tx{,}
\end{gather*}
which is a Jacobi form of weight and index~$\frac{1}{2}$ that also vanishes to order~$1$ at~$z = 0$. Its residue equals the cube of the Dedekind~$\eta$ function.

Thus we have the weakly holomorphic Jacobi form
\begin{gather*}
  \psi(\tau, z)
\defeq
  \eta^{-3 \nu}
  \theta(\tau, z)^\nu\,
  \phi(\tau, z)
\tx{,}
\end{gather*}
which has order~$0$ at~$z = 0$. The leading Fourier coefficients of~$\psi$ with respect to~$\tau$ coincide with its usual Fourier coefficients, and are thus Laurent polynomials in~$e(z \slash N')$, where~$N'$ is the level of the character of~$\psi$. Note that~$N$ and~$N'$ agree up to a power of~$2$. Now we compare the left and right hand side of
\begin{align*}
&
  \sum_{n \in \QQ}
  \big( e(\tfrac{1}{2} z) - e(- \tfrac{1}{2} z) \big)^{-\nu}\,
  \tdc(\phi; n; z)\,
  e(n \tau)
\\[.2\baselineskip]
={}&
  \phi(\tau, z)
=
  \eta^{3 \nu}
  \theta(\tau, z)^{-\nu}\,
  \psi(\tau, z)
=
  \eta^{3 \nu}
  \theta(\tau, z)^{-\nu}\,
  \sum_{n \in \QQ}
  \tdc(\psi; n; z)\,
  e(n \tau)
\\
={}&
  \Big(
  \prod_{n = 1}
  \big( 1 - e(n \tau) \big)^{2 \nu}
  \big( 1 - e(n \tau + z) \big)^{- \nu}
  \big( 1 - e(n \tau - z) \big)^{- \nu}
  \Big)
\\
&\qquad\qquad
  \sum_{n \in \QQ}
  \big( e(\tfrac{1}{2} z) - e(- \tfrac{1}{2} z) \big)^{-\nu}\,
  \tdc(\psi; n; z)\,
  e(n \tau)
\tx{.}
\end{align*}
When matching Fourier coefficients with respect to~$\tau$, we conclude that all~$\tdc(\phi;n;z)$ are Laurent polynomials in~$e(z \slash \lcm(2,N'))$, since all~$\tdc(\psi; n; z)$ are so. Using that~$\phi$ transforms like~$\chi$ under~$\rmH(\ZZ)$, if~$N$ is odd, we can remove any power of~$2$ from the denominator~$\lcm(2,N')$ in the proof, finishing the argument.
\end{proof}

To clarify the connection between the Fourier coefficients of~$\phi$ and~$f$ and to prepare Definition~\ref{def:explainable_ramanujan_type_congruence}, we offer the next statement.

\begin{proposition}%
\label{prop:cyclotomic_divides_fourier_implies_divisible_by_ell}
Consider a weakly holomorphic Jacobi form~$\phi$ with possible singularities at~$z \in \ZZ + \tau \ZZ$ that specializes to~$f$. Then we have~$c(f;n) = \tdc(\phi;n;0)$.

Further, if the Fourier coefficients of~$\phi$ are~$\ell$\nbd{}integral for a prime~$\ell$ and the~$\ell$\thdash{} cyclotomic polynomial in~$e(z)$ divides~$\tdc(\phi; n; z)$, then~$c(f; n) \equiv 0 \,\pmod{\ell}$.
\end{proposition}

\begin{proof}
Since~$e(z \slash 2) - e(-z \slash 2)$ has vanishing order~$1$ and leading Taylor coefficient~$1$ at~$z = 0$, we have
\begin{gather*}
  f(\tau)
=
  \lim_{z \ra 0}
  \big( e(\tfrac{1}{2} z) - e(- \tfrac{1}{2} z) \big)^\nu\,
  \phi(\tau, z)
\tx{.}
\end{gather*}
Next, we insert the expansion of~$\phi$ in~\eqref{eq:def:leading_fourier_coefficient} and obtain that
\begin{align*}
  f(\tau)
={}&
  \lim_{z \ra 0}
  \big( e(\tfrac{1}{2} z) - e(- \tfrac{1}{2} z) \big)^\nu\,
  \sum_{n \in \QQ}
  \big( e(\tfrac{1}{2} z) - e(- \tfrac{1}{2} z) \big)^{-\nu}\,
  \tdc(\phi; n; z)\,
  e(n \tau)
\\
={}&
  \lim_{z \ra 0}
  \sum_{n \in \QQ}
  \tdc(\phi; n; z)\,
  e(n \tau)
=
  \sum_{n \in \QQ}
  \lim_{z \ra 0}
  \tdc(\phi; n; z)\,
  e(n \tau)
  \sum_{n \in \QQ}
  \tdc(\phi; n; 0)\,
  e(n \tau)
\tx{.}
\end{align*}
We can exchange limit and summation, since the latter represents a holomorphic function and is therefore absolutely and locally uniformly convergent. We can perform the limit by inserting~$z$, since~$\tdc(\phi; n; 0)$ is a Laurent polynomial in~$e(z \slash 2)$ by Proposition~\ref{prop:fourier_coefficients_specialization}. This proves the first statement of the proposition when comparing the Fourier expansion of~$f$ against the right hand side.

To finish the proof, observe that if
\begin{gather*}
  \hat{c}(\phi; n; z)
\defeq
  \Big( \sum_{r = 0}^{\ell-1} e(r z) \Big)^{-1}\,
  \tdc(\phi; n; z)
\end{gather*}
is a Laurent polynomial in~$e(z \slash 2)$, by Gau\ss's Lemma it has $\ell$\nbd{}integral coefficients if and only if~$\tdc(\phi; n; z)$ does. In this situation, we conclude that
\begin{gather*}
  c(f; n)
=
  \tdc(\phi; n; 0)
=
  \Big( \sum_{r = 0}^{\ell-1} e(r 0) \Big)\,
  \hat{c}(\phi; n; 0)
=
  \ell\,
  \hat{c}(\phi; n; 0)
\tx{,}
\end{gather*}
where the second factor is~$\ell$\nbd{}integral and therefore~$c(f; n) \equiv 0 \,\pmod{\ell}$ as required.
\end{proof}

Proposition~\ref{prop:cyclotomic_divides_fourier_implies_divisible_by_ell} justifies the next definition.

\begin{definition}%
\label{def:explainable_ramanujan_type_congruence}
Let~$\ell$ be a prime, $M$ a positive integer, and~$\beta \in \QQ$. Further, let~$\phi$ be a weakly holomorphic Jacobi form with possible singularities at~$z \in \ZZ + \tau \ZZ$ that specializes to~$f$. If the Fourier coefficients of~$\phi$ are~$\ell$\nbd{}integral and the~$\ell$\thdash{} cyclotomic polynomial in~$e(z)$ divides~$\tdc(\phi; n; z)$ for all~$n \in M \ZZ + \beta$, we say that~$\phi$ explains the Ramanujan-type congruence modulo~$\ell$ of~$f$ on~$M \ZZ + \beta$.

Without further reference to~$f$, we say that~$\phi$ explains a Ramanujan-type congruence modulo~$\ell$ on~$M \ZZ + \beta$.
\end{definition}

When leveraging the theory of Jacobi forms to study explainable Ramanujan-like congruences, we will use analytic arguments and in particular the vanishing locus of~$\phi$. Towards this, we next reformulate Proposition~\ref{prop:cyclotomic_divides_fourier_implies_divisible_by_ell}, which features an~a~priori algebraic divisibility in~$\CC[e(z \slash N), e(-z)]$.

\begin{lemma}%
\label{la:cyclotomic_divides_equivalent_to_vanishing}
Let~$\ell$ be a prime, and~$\phi$ a weakly holomorphic Jacobi form with possible singularities at~$z \in \ZZ + \tau \ZZ$ that specializes to~$f$. Then the~$\ell$\thdash{} cyclotomic polynomial in~$e(z)$ divides~$\tdc(\phi; n; z)$ if and only if
\begin{gather*}
  c(\phi; n; z) = 0
  \quad\tx{for all\ }
  z \in \mfrac{1}{\ell} \ZZ \setminus \ZZ
\tx{.}
\end{gather*}
\end{lemma}

\begin{proof}
By the definition in~\eqref{eq:def:leading_fourier_coefficient} of the leading Fourier coefficients of~$\phi$, for positive but sufficiently small~$\Im(z)$, we have
\begin{gather*}
  c(\phi; n; z)
=
  \big( e(\tfrac{1}{2} z) - e(- \tfrac{1}{2} z) \big)^{-\nu}\,
  \tdc(\phi; n; z)
\tx{.}
\end{gather*}
As a function in~$z$, the first factor on the right hand side is regular and vanishes if and only if~$z \in \ZZ$. As a consequence, $c(\phi; n; z)$ vanishes for~$z \in \frac{1}{\ell} \ZZ \setminus \ZZ$ if and only if~$\tdc(\phi; n; z)$ does. By Proposition~\ref{prop:fourier_coefficients_specialization}, $\tdc(\phi; n; z)$ is a Laurent polynomial in~$e(z \slash N)$. Since the~$\ell$\thdash{} cyclotomic polynomial in~$e(z)$ vanishes exactly for~$z \in \frac{1}{\ell} \ZZ \setminus \ZZ$, we obtain the result.
\end{proof}

Applying Lemma~\ref{la:cyclotomic_divides_equivalent_to_vanishing} coefficient-wise, we obtain a characterization of explainable Ra\-ma\-nu\-jan-type congruence in analytic terms.

\begin{corollary}%
\label{cor:erc_characterization}
Let~$\ell$ be a prime, and~$\phi$ a weakly holomorphic Jacobi form with possible singularities at~$z \in \ZZ + \tau \ZZ$ that specializes to~$f$. Then~$\phi$ explains a Ramanujan-type congruence modulo~$\ell$ of~$f$ on~$M \ZZ + \beta$ if and only if the following function vanishes for all~$z \in \frac{1}{\ell} \ZZ \setminus \ZZ$:
\begin{gather*}
  \sum_{\substack{n \in \QQ\\ n \equiv \beta \pmod{M}}}
  c(\phi; n; z)\,
  e(n \tau)
\tx{.}
\end{gather*}
\end{corollary}

\begin{example}
We consider the Ramanujan-like congruences of~$2$\nbd{}colored partition functions that appear in work of Rolen-Tripp-Wagner~\cite{rolen-tripp-wagner-2020-preprint}. The weakly holomorphic Jacobi form with singularities at torsion points that explains it equals
\begin{gather*}
  \frac{\eta^4(\tau)}{\theta(\tau, 0)\, \theta(\tau, 2z)}
\tx{.}
\end{gather*}
Its singularities lie at~$\frac{1}{2} (2\ZZ + \tau 2 \ZZ) = \ZZ + \tau \ZZ$. Observe that Rolen-Tripp-Wagner include normalizing factors in their definition to make the Fourier indices that appear match the natural combinatorial indexing, while we omit them to preserve modular properties.
\end{example}

We close this section by expressing the Fourier expansion in Corollary~\ref{cor:erc_characterization} in terms of the slash action of~$\GMp{1}^+(\ZZ)$. Observe that we have an action of~$\GL{2}^+(\RR)$ on~$\RR^2$, so that we can extend the metaplectic Jacobi group~$\Mp{1}(\RR) \ltimes \rmH(\RR)$ to similitudes: $\GMp{1}^+(\RR) \ltimes \rmH(\RR)$. Lemma~\ref{la:fourier_expansion_projection_erc} is an analogue of Proposition~2.4 in~\cite{raum-2022} and~\eqref{eq:fourier_expansion_projection_rc} in this work.

\begin{lemma}%
\label{la:fourier_expansion_projection_erc}
Assume that~$\phi$ is a weakly holomorphic Jacobi form with singularities at torsion points. Given~$M \in \ZZ_{\ge 1}$ and $\beta \in \QQ$, assume that the Fourier expansion of~$\phi$ with respect to~$\tau$ is supported on~$\beta + \ZZ$. Then we have the Fourier expansion
\begin{gather}
\label{eq:la:fourier_expansion_projection}
  M^{\frac{k}{2}-1}\,
  \sum_{h \pmod{M}}
  e\big( \mfrac{-\beta h}{M} \big)
  \big(
  \phi \big|_{k,m}\,
  \begin{psmatrix} 1 & h \\ 0 & M \end{psmatrix}
  \big)(\tau, z)
=
  \sum_{\substack{n \in \beta + \ZZ\\ n \equiv \beta \pmod{M}}}
  c\big( \phi; n; \mfrac{z}{M} \big)\,
  e\big( \mfrac{n}{M} \tau \big)
\tx{.}
\end{gather}
In particular, the summation over~$h \,\pmod{M}$ on the left hand side is well-defined.

Further, if~$\mu \in \QQ$ is a regular torsion point of\/~$\phi$, we have
\begin{gather}
\label{eq:la:fourier_expansion_projection:specialization}
  M^{\frac{k}{2}-1}\,
  \sum_{h \pmod{M}}
  e\big( \mfrac{-\beta h}{M} \big)
  \big(
  \phi \big|_{k,m}\,
  (0, \mu)\;
  \big|_{k,m}\,
  \begin{psmatrix} 1 & h \\ 0 & M \end{psmatrix}
  \big)(\tau, 0)
=
  \sum_{\substack{n \in \beta + \ZZ\\ n \equiv \beta \pmod{M}}}
  c\big( \phi; n; \mu \big)\,
  e\big( \mfrac{n}{M} \tau \big)
\tx{.}
\end{gather}
\end{lemma}

\begin{proof}
Note that if~$\Im(z) \slash \Im(\tau) > 0$ is sufficiently small, then so is~$\Im(z \slash M) \slash \Im((\tau + h) \slash M) > 0$. We can therefore insert the Fourier expansion~\eqref{eq:def:fourier_expansion_jacobi} in the following calculation:
\begin{align*}
&
  \sum_{h \pmod{M}}
  e\big( \mfrac{-\beta h}{M} \big)
  \big(
  \phi \big|_{k,m}\,
  \begin{psmatrix} 1 & h \\ 0 & M \end{psmatrix}
  \big)(\tau, z)
\\
={}&
  M^{-\frac{k}{2}}\,
  \sum_{h \pmod{M}}
  e\big( \mfrac{-\beta h}{M} \big)
  \phi\big(
  \mfrac{\tau + h}{M}, \mfrac{z}{M}
  \big)
=
  M^{-\frac{k}{2}}\,
  \sum_{n \in \QQ}
  \sum_{h \pmod{M}}
  e\big( \mfrac{h (n - \beta)}{M} \big)
  c\big( \phi; n; \mfrac{z}{M} \big)\,
  e\big( \mfrac{n}{M} \tau \big)
\tx{.}
\end{align*}
Note that~$n - \beta \in \ZZ$ for nonzero Fourier coefficients by our assumptions. Hence when evaluating the exponential sum we obtain the desired condition~$n \equiv \beta \,\pmod{M}$ and an additional factor~$M$, which appears in~\eqref{eq:la:fourier_expansion_projection}, confirming that equality.

To show~\eqref{eq:la:fourier_expansion_projection:specialization}, it now suffices to note that in~$\Mp{1}(\RR) \ltimes \Heis(\RR)$ we have
\begin{gather*}
 (0,\mu)\, \begin{psmatrix} 1 & h \\ 0 & M \end{psmatrix} 
=
 \begin{psmatrix} 1 & h \\ 0 & M \end{psmatrix}\, (0, M \mu)
\tx{,}
\end{gather*}
and that the factor of automorphy contributed by~$(0, M \mu)$ to the slash action is trivial.
\end{proof}

\section{Reformulation into representation theory}%
\label{sec:representation_theory}

In Corollary~\ref{cor:erc_characterization} we arrived at an analytic characterization of Ramanujan-type congruences that are ``explained'' by some fixed weakly holomorphic Jacobi form with singularities at torsion points. The goal of this section is to reformulate it in a purely representation theoretic way, which only refers to~$\phi$ via its character~$\chi$. We subdivide our discussion into three parts.

In Section~\ref{ssec:representation_theory:homomorphisms} we provide a homomorphism between a double-induction and a space of weakly holomorphic modular forms that naturally arises from Lemma~\ref{la:fourier_expansion_projection_erc}. This homomorphism, much like the abstract spaces of weakly holomorphic modular forms from previous work, enables a characterization of explainable Ramanujan-type congruences that we present in Section~\ref{ssec:representation_theory:characterization}. Finally, Section~\ref{ssec:representation_theory:square_classes} provides the representation theoretic argument for our main theorem.

\subsection{Homomorphisms}%
\label{ssec:representation_theory:homomorphisms}

When combining Corollary~\ref{cor:erc_characterization} with Lemma~\ref{la:fourier_expansion_projection_erc}, a representation theoretic perspective naturally leads to the question of how the left hand side of~\eqref{eq:la:fourier_expansion_projection:specialization} behaves under the slash action of the special linear group and its metaplectic cover. To answer this question, we present two lemmas that reveal the relevant representation in steps.

To omit auxiliary assumptions about regularity in these lemmas, we need a modified notion of constant terms with respect to~$z$ at torsion points. Specifically, if~$\phi$ has singularities at torsion points the order~$\ord_{z_0}(\phi)$ of~$\phi$ at~$z_0 = \la \tau + \mu$ is well-defined, and we set
\begin{gather*}
  \const_{z_0}(\phi)(\tau)
\defeq
  \begin{cases}
    \phi(\tau, z_0) \tx{,} & \tx{if\ }\ord_{z_0}(\phi) \ge 0
  \tx{;}
  \\
    0\tx{,} & \tx{if\ }\ord_{z_0}(\phi) < 0
  \tx{.}
  \end{cases}
\end{gather*}
Our first lemma asserts that the modified constant term at~$0$ is a homomorphism of~$\Mp{1}(\RR)$ representations.

\begin{lemma}%
\label{la:modified_constant_term_homomorphism}
Let~$\phi$ be a meromorphic function on~$\HS \times \CC$ with singularities at torsion points. Then for~$k \in \frac{1}{2} \ZZ$ and~$m \in \CC$, we have
\begin{gather*}
  \const_0\big( \phi \big|_{k,m}\, (g,\om) \big)
=
  \const_0(\phi) \big|_k\, (g,\om)
\quad
  \tx{for all\ } (g,\om) \in \GMp{1}^+(\RR)
\tx{.}
\end{gather*}
\end{lemma}

\begin{proof}
Note that the order of~$\phi$ at~$0$ does not change when applying~$(g,\om)$. It thus suffices to note that the factors of automorphy associated with~$(g,\om)$ with~$g = \begin{psmatrix} a & b \\ c & d \end{psmatrix}$ on the left and right hand side relate in the following way:
\begin{gather*}
  \det(g)^{\frac{k}{2}}
  \om(\tau)^{-2 k}
  e\big( m \mfrac{- c z^2}{c \tau + d} \big)
  \Big|_{z = 0}
=
  \det(g)^{\frac{k}{2}}
  \om(\tau)^{-2 k}
\tx{.}
\end{gather*}
\end{proof}

We next accommodate the torsion points~$(0,\mu)$ that appear in~\eqref{eq:la:fourier_expansion_projection:specialization}. The case that is relevant in the context of Corollary~\ref{cor:erc_characterization} is~$\mu \in \frac{1}{\ell} \ZZ \setminus \ZZ$. We restrict ourselves to this for simplicity.

\begin{lemma}%
\label{la:torsion_point_representation}
Suppose that~$V$ is a space of meromorphic functions on~$\HS \times \CC$ that is stable under the weight~$k \in \frac{1}{2}$ and index~$m \in \CC$ slash action of\/~$\Ga \ltimes \Heis(\ZZ)$ for a subgroup~$\Ga \subseteq \Mp{1}(\ZZ)$. Then given a prime~$\ell$, we have the homomorphism of\/~$\Ga$\nbd{}representations
\begin{gather}
\label{eq:la:torsion_point_representation}
\begin{aligned}
  \Ind_{\Ga \cap \tdGa_1(\ell)}^\Ga\,
  \iota^\ast(V)
&\lra
  \linspan_\CC\big\{
  \phi |_{k,m}\, (\la,\mu)
  \condcol
  \phi \in V,\ 
  (\la,\mu) \in \tfrac{1}{\ell} \ZZ^2 \setminus \ZZ^2
  \big\}
\tx{,}
\\ 
  \phi \otimes (\ga,\om)
&\lmto
  \phi \big|_{k,m}\, \big( 0, \tfrac{1}{\ell} \big) (\ga,\om)
\tx{,}
\end{aligned}
\end{gather}
where~$\iota^\ast(V)$ is the pullback of\/~$V$ along the embedding
\begin{gather*}
  \iota
=
  \iotator{\ell}
\defcol
  \tdGa_1(\ell)
\lhra
  \tdGa_1(\ell) \ltimes \Heis(\ZZ)
\tx{,}\quad
  (\ga,\om)
\lmto
  \big( (\ga, \om), \tfrac{c}{\ell}, \tfrac{d-1}{\ell}, 0 \big)
\quad\tx{with\ }
  \ga
=
  \begin{psmatrix} a & b \\ c & d \end{psmatrix}
\tx{.}
\end{gather*}
\end{lemma}

\begin{proof}
Note that~$(0, \frac{1}{\ell}) \ga \in \tfrac{1}{\ell} \ZZ^2 \setminus \ZZ^2$ for all~$\ga \in \SL{2}(\ZZ)$. This shows that the image of~$\phi \otimes (\ga,\om)$ given in~\eqref{eq:la:torsion_point_representation} lies in the stated co-domain. To verify that the map is well-defined and yields a homomorphism note that for~$(\ga,\om) \in \Ga \cap \tdGa_1(\ell)$, using notation for~$\ga$ as in the statement, we have 
\begin{gather*}
  \phi \big|_{k,m}\, (0, \tfrac{1}{\ell}) (\ga,\om)
=
  \Big(
  \phi
  \big|_{k,m}\,
  \big( (\ga, \om), \mfrac{c}{\ell}, \mfrac{d-1}{\ell}, 0 \big)
  \Big) \Big|_{k,m}\, (0, \tfrac{1}{\ell})
\tx{,}
\end{gather*}
which follows when multiplying out the group elements on the left and right hand side.
\end{proof}

Let~$M$ be a positive integer and~$\ell$ be a prime. We consider a character~$\chi$ of~$\Ga \ltimes \Heis(\ZZ)$ for subgroup~$\Ga \subseteq \Mp{1}(\ZZ)$. Combining the previous two lemmas with Lemma~\ref{la:hecke_representation}, which provides an understanding of the action of~$\GMp{1}^{(M)}(\ZZ)$, we obtain in the next statement a homomorphism from
\begin{gather}
\label{eq:def:hecke_torsion_representation}
  \HeckeTor{M,\ell}(\chi)
\defeq
  \Ind_{\delta_M^{-1} \Ga \delta_M \cap \Mp{1}(\ZZ)}^{\Mp{1}(\ZZ)}\,
  \iotahecke{M}^\ast\Big(
  \Ind_{\Ga \cap \tdGa_1(\ell)}^{\Ga}\,
  \iotator{\ell}^\ast(\chi)
  \Big)
\tx{,}
\end{gather}
where the maps~$\iotahecke{M}$ and~$\iotator{\ell}$ are the embeddings given in Lemmas~\ref{la:hecke_representation} and~\ref{la:torsion_point_representation}.

\begin{proposition}
\label{prop:hecke_torsion_representation}
Let~$\phi$ be a weakly holomorphic Jacobi form with singularities at torsion points of weight~$k \in \frac{1}{2} \ZZ$ and index~$m \in \CC$ for a character~$\chi$ of\/~$\Ga \ltimes \Heis(\ZZ)$ with~$\Ga \subseteq \Mp{1}(\ZZ)$. Then given a positive integer~$M$ and a prime~$\ell$, we have the homomorphism of\/~$\Ga$\nbd{}representations
\begin{gather}
\label{eq:prop:hecke_torsion_representation}
\begin{aligned}
  \HeckeTor{M,\ell}(\chi)
&\lra
  \linspan_\CC\Big\{
  \const_{0} \big(
  \phi |_{k,m}\, (\la, \mu)
  \big)
  \big|_{k,m}\, \delta 
  \condsep
  \la, \mu \in \tfrac{1}{\ell} \ZZ,\ 
  \delta \in \GMp{1}^{(M)}(\ZZ)
  \Big\}
\tx{,}
\\
  (\ga',\om') \otimes (\ga,\om)
&\lmto
  \const_0\Big(
  \phi
  \big|_{k,m}\, \big(0, \tfrac{1}{\ell}\big) (\ga',\om')
  \Big)
  \Big|_{k,m}\, \begin{psmatrix} 1 & 0 \\ 0 & M \end{psmatrix} (\ga,\om)
\tx{.}
\end{aligned}
\end{gather}
\end{proposition}

\begin{proof}
This follows when combining Lemmas~\ref{la:hecke_representation} and~\ref{la:torsion_point_representation} with Lemma~\ref{la:modified_constant_term_homomorphism}.
\end{proof}

\subsection{Characterization of explainable Ramanujan-type congruences}%
\label{ssec:representation_theory:characterization}

We note that we up till this point did not need to make any assumptions on~$\Ga \subseteq \Mp{1}(\ZZ)$ nor the character~$\chi$. We next need to provide vectors in~$\HeckeTor{M,\ell}(\chi)$ that correspond to Ramanujan-type congruence via the formula in Lemma~\ref{la:fourier_expansion_projection_erc}. To this end, we need to ensure that~$\HeckeTor{M,\ell}(\chi)$ is sufficiently large.

Let\/~$\Ga \subseteq \Mp{1}(\ZZ)$ be a subgroup. In Section~\ref{sec:ramanujan_type}, we already used that without loss of generality~$T \in \Ga$ if the Fourier expansion of~$f$ is supported on~$\beta + \ZZ$ for some~$\beta \in \QQ$. This is sufficient to conclude that 
\begin{gather}
\label{eq:la:subgroup_index_large_hecke_torsion_representation:hecke}
  \begin{psmatrix} 1 & \ZZ \\ 0 & M \end{psmatrix}
\subseteq
  \begin{psmatrix} 1 & 0 \\ 0 & M \end{psmatrix}
  \Ga
\tx{,}
\end{gather}
which ensures that~$\Hecke{M}(\chi)$ is sufficiently large to contain the vectors~$\hecke{M}(\chi; \beta)$. The next lemma provides an analogous sufficient criteria for the representation in Lemma~\ref{la:torsion_point_representation}.

\begin{lemma}%
\label{la:subgroup_index_gives_large_hecke_torsion_representation}
Given a prime~$\ell$, suppose that the index of\/~$\Ga \cap \tdGa_1(\ell)$ in~$\Ga \cap \tdGa_0(\ell)$ equals~$\ell - 1$. Then we have
\begin{gather}
\label{eq:la:subgroup_index_large_hecke_torsion_representation:torsion}
  \big( 0, \tfrac{1}{\ell} \ZZ \setminus \ZZ \big)
\subseteq
  \ZZ^2 + \big( 0, \tfrac{1}{\ell} \big) \Ga
\tx{.}
\end{gather}
This assumption is satisfied if\/~$\Ga$ is a level~$N$ congruence subgroup with~$\gcd(\ell, N) = 1$.
\end{lemma}

\begin{proof}
It suffices to calculate in~$\SL{2}(\ZZ)$. Employing the Orbit Theorem, we deduce the first statement from the observation that
\begin{gather*}
  \big( (0, \tfrac{1}{\ell} \ZZ) \setminus (0,\ZZ) \big) \big\slash (0,\ZZ)
\end{gather*}
has size~$\ell - 1$ and that the stabilizer of~$(0, \frac{1}{\ell})$ in~$\Ga_0(\ell)$ equals~$\Ga_1(\ell)$. The assertion on congruence subgroups follows from the Chinese Remainder Theorem.
\end{proof}

We can now provide the connection between the Fourier expansion in~\eqref{eq:la:fourier_expansion_projection:specialization}
and the representation in Proposition~\ref{prop:hecke_torsion_representation}. In analogy with the definition of~$\hecke{M}(\chi; \beta)$ in~\eqref{eq:def:hecke_vector}, we assume that~$T \in \Ga$, $\chi(T) = e(\beta)$, and that~\eqref{eq:la:subgroup_index_large_hecke_torsion_representation:torsion} holds. Then for~$\mu \in \frac{1}{\ell} \ZZ \setminus \ZZ$, we set
\begin{gather}
\label{eq:def:hecke_torsion_vector}
  \hecketor{M,\ell}(\chi; \beta, \mu)
\defeq
  \ga_\mu\,
  M^{1 - \frac{k}{2}}\mspace{-12mu}
  \sum_{h \pmod{M}}
  e\big( \mfrac{- h \beta}{M} \big)\, T^h
\mathrel{\;\in\;}
  \HeckeTor{M,\ell}(\chi)
\tx{,}
\end{gather}
where~$(0, \frac{1}{\ell}) \ga_\mu = \big( 0, \mu \big)$ and~$\ga_\mu \in \Ga$. Note that this vector depends on the choice of~$\ga_\mu$, which we suppress from our notation, only up to a nonzero scalar multiple.

\begin{lemma}%
\label{la:fourier_expansion_projection_in_image}
Under the assumptions of~\eqref{eq:def:hecke_torsion_vector}, $\hecketor{M,\ell}(\chi; \beta, \mu)$ maps to a nonzero multiple of the left hand side of~\eqref{eq:la:fourier_expansion_projection:specialization} under the homomorphism in Proposition~\ref{prop:hecke_torsion_representation}.
\end{lemma}

\begin{proof}
This is a direct consequence of~\eqref{eq:la:subgroup_index_large_hecke_torsion_representation:torsion} and~\eqref{eq:la:subgroup_index_large_hecke_torsion_representation:hecke} and the formula for the assignment given in~\eqref{eq:prop:hecke_torsion_representation}.
\end{proof}

With Lemma~\ref{la:fourier_expansion_projection_in_image} at hand, we are in position to characterize explainable Ramanujan-type congruences. To lighten technical aspects, we will from now on assume that~$\Ga$ is a congruence subgroup, and maintain the assumption from Section~\ref{sec:ramanujan_type} that~$T \in \Ga$.

\begin{proposition}
\label{prop:characterization_explainable_ramanujan_type_congruence}
Consider a weakly holomorphic Jacobi form~$\phi$ with possible singularities at~$z \in \ZZ + \tau \ZZ$ for a character~$\chi$ of\/~$\Ga \ltimes \Heis(\ZZ)$, where~$\Ga$ is a congruence subgroup that contains~$T$. Let~$\ell$ be a prime, $M$ a positive integer, and~$\beta$ a rational number with~$\chi(T) = e(\beta)$.

Then~$\phi$ explains a Ramanujan-type congruence modulo~$\ell$ on~$M \ZZ + \beta$ if and only if
\begin{gather*}
  \hecketor{M,\ell}(\chi;\beta,\mu)
\in
  \ker(\psi)
\quad
  \tx{for all\ } \mu \in \tfrac{1}{\ell}\ZZ \setminus \ZZ
\tx{,}
\end{gather*}
where~$\psi$ is the homomorphism in Proposition~\ref{prop:hecke_torsion_representation}.
\end{proposition}

\begin{proof}
This is a consequence of Lemma~\ref{la:fourier_expansion_projection_in_image}, the formula in Lemma~\ref{la:fourier_expansion_projection_erc}, and Corollary~\ref{cor:erc_characterization}.
\end{proof}

\subsection{Square-classes of explainable Ramanujan-type congruences}%
\label{ssec:representation_theory:square_classes}

Recall Theorem~\ref{thm:square_classes_rc} on square-classes~$M \ZZ + a^2 \beta$ of arithmetic progression on which we have Ramanujan-type congruences of a fixed weakly holomorphic modular form~$f$. In this section, we establish the analogue for explainable Ramanujan-type congruence, and provide the main corollary that restricts possible~$M$ and~$\beta$ for maximal progressions that support explainable Ramanujan-type congruences.

The next lemma reexamines the torsion point representation in Lemma~\ref{la:torsion_point_representation}. It allows us to later employ the same techniques as for usual Ramanujan-type congruences.

\begin{lemma}%
\label{la:torsion_point_representation_ga0ell}
Let~$\chi$ be a character of\/~$\Ga$ for~$\Ga \subseteq \tdGa_0(\ell)$. Assume that~$T \in \Ga$ and~$\Ga$ has index~$\ell - 1$ in~$\Ga \cap \tdGa_1(\ell)$. Viewing~$\chi$ as a character of\/~$\Ga \ltimes \rmH(\ZZ)$, we have
\begin{gather*}
  \Ind^\Ga_{\Ga \cap \tdGa_1(\ell)}\,
  \iotator{\ell}
\cong
  \bigoplus_{d \pmod{\ell}^\times}
  \chi_d
\tx{,}
\end{gather*}
where~$\iotator{\ell}$ is as in Lemma~\ref{la:torsion_point_representation} and~$\chi_d$ is the character with
\begin{gather*}
  \Res_{\Ga \cap \tdGa_1(\ell)}^\Ga\, \chi_d
\cong
  \Res_{\Ga \cap \tdGa_1(\ell)}^\Ga\, \iotator{\ell}^\ast( \chi )
\tx{.}
\end{gather*}
\end{lemma}

\begin{remark}
We make no claim about which characters~$\chi_d$ of~$\Ga$ appear and how often. Although the proof shows that~$\chi_d$ is a twist of~$\chi$ by (the lift to~$\tdGa_0(\ell)$) a Dirichlet character modulo~$\ell$, it does not exhibit multiplicities.
\end{remark}

\begin{proof}
We first determine the restriction to~$\Ga \cap \tdGa_1(\ell)$ via Mackey's restriction theorem. By our assumption on~$\Ga$, we have a set of representatives for
\begin{gather*}
  \big( \Ga \cap \tdGa_1(\ell) \big) \big\backslash \Ga \big\slash \big( \Ga \cap \tdGa_1(\ell) \big)
\end{gather*}
that comprises elements~$\ga_d \in \Ga$ with bottom right entry~$d$ running through the units modulo~$\ell$. Fixing such a traversal, we find that
\begin{gather*}
  \Res^\Ga_{\Ga \cap \tdGa_1(\ell)}\,
  \Ind^\Ga_{\Ga \cap \tdGa_1(\ell)}\,
  \iotator{\ell}^\ast(\chi)
\cong
  \bigoplus_{d \pmod{\ell}^\times}\,
  \Ind^{\Ga \cap \tdGa_1(\ell)}_{\Ga'_d}\,
  \iota^{\prime\,\ast}_d\big(
  \iotator{\ell}^\ast(\chi)
  \big)
\tx{,}
\end{gather*}
where
\begin{gather*}
  \Ga'_d
\defeq
  \ga_d^{-1}
  \big( \Ga \cap \tdGa_1(\ell) \big)
  \ga_d
  \cap
  \big( \Ga \cap \tdGa_1(\ell) \big)
\tx{,}\quad
  \iota'_d(\ga)
\defeq
  \ga_d
  \ga
  \ga_d^{-1}
\tx{.} 
\end{gather*}
Using our assumptions on~$\Ga$ yield~$\Ga'_d = \Ga \cap \tdGa_1(\ell)$ and
\begin{gather*}
  \iota^{\prime\,\ast}_d\big(
  \iotator{\ell}^\ast(\chi)
  \big)
=
  \iotator{\ell}^\ast(\chi)
\tx{.}
\end{gather*}
This shows that
\begin{gather*}
  \Res^\Ga_{\Ga \cap \tdGa_1(\ell)}\,
  \Ind^\Ga_{\Ga \cap \tdGa_1(\ell)}\,
  \iotator{\ell}^\ast(\chi)
\cong
  \bigoplus_{d \pmod{\ell}^\times}\,
  \iotator{\ell}^\ast(\chi)
\tx{.}
\end{gather*}

We next use the assumption that~$\chi$, which we view as a character of~$\Ga \ltimes \rmH(\ZZ)$ via the projection to~$\Ga$, is trivial on~$\rmH(\ZZ)$. This implies that~$\iotator{\ell}^\ast(\chi) \cong \chi$. We conclude that
\begin{gather*}
  \Res^\Ga_{\Ga \cap \tdGa_1(\ell)} \Big(
  \ov\chi \otimes
  \Ind^\Ga_{\Ga \cap \tdGa_1(\ell)}\,
  \iotator{\ell}^\ast(\chi)
  \Big)
\end{gather*}
is isotrivial. Therefor all constituents of the unrestricted representation appear in the induction
\begin{gather*}
  \Ind^\Ga_{\Ga \cap \tdGa_1(\ell)}\, \bbone
\cong
  \Ind^{\Mp{1}(\ZZ)}_{\tdGa_1(\ell)}\, \bbone
\tx{,}
\end{gather*}
which is a direct sum of characters, commonly identified with the Dirichlet characters modulo~$\ell$.
\end{proof}

We are now able to proof our main theorem. Note that as opposed to Theorem~\ref{thm:square_classes_rc}, we do not need to make the assumption that the kernel of~$\chi$ be a congruence subgroup, as we do not rely on Theorem~\ref{thm:mf_representation}. It is unclear, however, whether there are any corresponding examples of explainable Ramanujan-type congruences. Recall the group ~$\Ga^\Delta_M(\chi)$ defined in~\eqref{eq:def:gamma_diag_char}.

\begin{theorem}%
\label{thm:square_classes_erc}
Let~$\chi$ be a character of\/~$\Ga \subseteq \Mp{1}(\ZZ)$, and view it as a character of\/~$\Ga \ltimes \Heis(\ZZ)$. Assume that~$T \in \Ga$ and let~$\beta \in \QQ$ be such that~$\chi(T) = e(\beta)$. Further, let~$\phi$ be a weakly holomorphic Jacobi form for~$\chi$ with possible singularities at~$\ZZ + \tau \ZZ$.

If\/~$\phi$ explains a Ramanujan-type congruence modulo~$\ell$ on~$M \ZZ + \beta$, then it explains such congruences on~$M \ZZ + a^2 \beta$ for all~$a \in \ZZ$ such that there is~$\begin{psmatrix} a & b \\ c & d \end{psmatrix} \in \Ga^{\Delta}_M(\chi) \cap \tdGa_0(\ell)$.
\end{theorem}

We will finish the paper with the proof of Theorem~\ref{thm:square_classes_erc}. Before that we deduce Theorem~\ref{mainthm:square_classes_erc} and Corollary~\ref{maincor:max_erc}. The former is essentially a special case of Theorem~\ref{mainthm:square_classes_erc}.

\begin{proof}[Proof of Theorem~\ref{mainthm:square_classes_erc}]
The proof of Corollary~\ref{cor:square_classes_rc_eta_theta} extends verbatime.
\end{proof}

\begin{proof}[Proof of Corollary~\ref{maincor:max_erc}]
The strategy of proof has appeared in various forms in~\cite{raum-2022} and~\cite{raum-2023}. For convenience, we revisit the argument.

By contraposition, we assume that~$\beta$ is~$p$\nbd{}integral, and~$\ord_p(M) \ge \ord_p(\beta) + 2$ if~$p$ is odd and~$\ord_2(M) \ge \ord_2(\beta) + 4$ if~$p = 2$. We claim that for every~$h \in \ZZ$ there is~$u \in \ZZ$ with~$u^2 \beta \equiv \beta + h M \slash p$ modulo~$M$. This follows by the Chinese Remainder Theorem if we show the later congruence modulo~$M_p$, where~$M_p$ is the maximal power of~$p$ dividing~$M$. We set~$\beta_p = \gcd(M_p, \beta)$ and define~$h'$ by~$h' (\beta \slash \beta_p) \equiv h \,\pmod{M_p}$.

The square root of~$1 + h' M \slash (p \beta_p)$ exists modulo any power of~$p$, since for odd~$p$ we have~$p \isdiv (M \slash p \beta_p)$, while for~$p = 2$ we have~$8 \isdiv (M \slash p \beta_p)$. That is, we can define~$u$ up to sign by the condition
\begin{gather*}
  u^2
\equiv 
  1 + h' M \slash (p \beta_p)
  \;\pmod{M_p \slash \gcd(M_p, \beta)}
\tx{.}
\end{gather*}
Now calculating the product~$u^2 \beta$ and inserting the definition of~$h'$ yields the statement.
\end{proof}

\begin{proof}[Proof of Theorem~\ref{thm:square_classes_erc}]
Replacing~$\Ga$ by~$\Ga \cap \tdGa_0(\ell)$, we can assume that~$\Ga \subset \tdGa_0(\ell)$. Fix~$\ga$ as in the statement, and let~$\psi$ be the homomorphism in Proposition~\ref{prop:hecke_torsion_representation}. In analogy with the proof of Theorem~\ref{thm:square_classes_rc}, by Proposition~\ref{prop:characterization_explainable_ramanujan_type_congruence} we know that
\begin{gather*}
  \hecketor{M,\ell}(\chi; \beta,\mu)
\in
  \ker(\psi)
\quad
  \tx{for all\ } \mu \in \tfrac{1}{\ell}\ZZ \setminus \ZZ
\tx{,}
\end{gather*}
an need to show that
\begin{gather*}
  \hecketor{M,\ell}(\chi; a^2 \beta,\mu)
\in
  \ker(\psi)
\quad
  \tx{for all\ } \mu \in \tfrac{1}{\ell}\ZZ \setminus \ZZ
\tx{.}
\end{gather*}

The inner representation in the definition~\eqref{eq:def:hecke_torsion_representation} by Lemma~\ref{la:torsion_point_representation_ga0ell} is a direct sum of characters~$\chi \chi_d$ where~$\chi_d$ is a Dirichlet character modulo~$\ell$. We therefore find that
\begin{gather*}
  \HeckeTor{M,\ell}(\chi)
\cong
  \bigoplus_{d \pmod{\ell}^\times}
  \Hecke{M,\ell}(\chi \chi_d)
\tx{.}
\end{gather*}
We identifying~$\ker(\psi)$ with the corresponding submodule of the right hand side.

Examining the definition of~$\hecketor{M,\ell}(\chi; \beta,\mu)$, we find that we have
\begin{gather*}
  \linspan\big\{
  \hecketor{M,\ell}(\chi; \beta,\mu)
  \condcol
  \mu \in \tfrac{1}{\ell}\ZZ \setminus \ZZ
  \big\}
\subseteq
  \HeckeTor{M,\ell}(\chi)
\end{gather*}
in the decomposition of~$\HeckeTor{M,\ell}(\chi)$ corresponds to
\begin{gather*}
  \bigoplus_{d \pmod{\ell}^\times}
  \CC \hecketor{M,\ell}(\chi \chi_d; \beta)
\subseteq
  \bigoplus_{d \pmod{\ell}^\times}
  \Hecke{M,\ell}(\chi \chi_d)
\tx{.}
\end{gather*}

Now, we apply the same argument as in the proof of Theorem~\ref{thm:square_classes_rc} to show that
\begin{gather*}
  \bigoplus_{d \pmod{\ell}^\times}
  \CC \hecketor{M,\ell}(\chi \chi_d; a^2 \beta)
\subseteq
  \ker(\psi)
\subseteq
  \bigoplus_{d \pmod{\ell}^\times}
  \Hecke{M,\ell}(\chi \chi_d)
\tx{.}
\end{gather*}
Reverting the decomposition of~$\HeckeTor{M,\ell}(\chi)$, we find that
\begin{gather*}
  \linspan\big\{
  \hecketor{M,\ell}(\chi; a^2\beta,\mu)
  \condcol
  \mu \in \tfrac{1}{\ell}\ZZ \setminus \ZZ
  \big\}
\subseteq
  \HeckeTor{M,\ell}(\chi)
\tx{,}
\end{gather*}
which is what we had to show.
\end{proof}

\ifbool{nobiblatex}{%
  \bibliographystyle{alpha}%
  \bibliography{bibliography.bib}%
  \addcontentsline{toc}{section}{References}
  \markright{References}
}{%
  \vspace{1.5\baselineskip}
  \phantomsection
  \addcontentsline{toc}{section}{References}
  \markright{References}
  \label{sec:references}
  \sloppy
  \printbibliography[heading=none]%
}

\Needspace*{3\baselineskip}
\noindent%
\rule{\textwidth}{0.15em}
\\

{\small\noindent
Martin Raum\\
Chalmers tekniska högskola och G\"oteborgs Universitet\\
Institutionen f\"or Matematiska vetenskaper\\
SE-412 96 G\"oteborg, Sweden\\
E-mail: \url{martin@raum-brothers.eu}\\
Homepage: \url{https://martin.raum-brothers.eu}
}%

\ifdraft{%
\listoftodos%
}

\end{document}

